\documentclass[12pt]{amsart}

\usepackage[a4paper,margin=22mm, top=24mm]{geometry}
\usepackage{amsmath,amssymb,mathtools,amsthm,amscd}
\usepackage[T1]{fontenc}
\usepackage{newtxtext,newtxmath,mathrsfs}
\usepackage[colorlinks=true,linkcolor=teal,citecolor=magenta,urlcolor=blue]{hyperref}
\usepackage{enumerate}

\newtheorem{theorem}{Theorem}[section]
\newtheorem{lemma}[theorem]{Lemma}
\newtheorem{proposition}[theorem]{Proposition}
\newtheorem{corollary}[theorem]{Corollary}
\theoremstyle{definition}
\newtheorem{definition}[theorem]{Definition}
\newtheorem{remark}[theorem]{Remark}

\DeclareMathOperator{\Tr}{Tr}
\DeclareMathOperator{\End}{End}
\DeclareMathOperator{\Hom}{Hom}
\DeclareMathOperator{\Der}{Der}
\DeclareMathOperator{\GL}{GL}
\DeclareMathOperator{\ad}{ad}
\DeclareMathOperator{\rank}{rank}
\DeclareMathOperator{\Spec}{Spec}
\DeclareMathOperator{\Span}{Span}
\DeclareMathOperator{\Rep}{Rep}
\DeclareMathOperator{\coker}{coker}
\DeclareMathOperator{\im}{im}
\DeclareMathOperator{\ch}{ch}
\DeclareMathOperator{\rad}{rad}
\DeclareMathOperator{\sm}{sm}

\newcommand{\red}{\mathrm{red}}
\newcommand{\inc}{\mathrm{inc}}
\newcommand{\fsl}{\mathfrak{sl}}
\newcommand{\aniso}{\mathrm{aniso}}

\newcommand{\Sym}{\mathrm{Sym}}
\newcommand{\Ass}{\mathsf{Assoc}}
\newcommand{\Var}{\mathsf{Var}}
\newcommand{\Inc}{\mathsf{Inc}}
\newcommand{\Leib}{\mathsf{Leib}}
\newcommand{\Comm}{\mathsf{Comm}}
\newcommand{\Lie}{\mathsf{Lie}}
\newcommand{\Type}{\mathsf{Type}}

\numberwithin{equation}{section}

\title{Geometric Rigidity in Moduli Stacks of Algebras}
\author{Atabey Kaygun}
\email{kaygun@itu.edu.tr}
\address{Istanbul Technical University, Istanbul, Turkey}

\begin{document}

\begin{abstract}
  We study quadratic moduli schemes $X$ of algebra laws on a fixed vector space $W$ under the
  transport-of-structure action of $\GL(W)$ on $\Hom(W^{\otimes 2},W)$. We construct an intrinsic
  three-term deformation complex on $X$ whose fibers encode transverse first-order classes and
  primary obstructions, and whose cohomology agrees on the operadic loci with the standard
  low-degree deformation cohomology (à la Gerstenhaber and Nijenhuis--Richardson). We then define a
  canonical quadratic map
  $\kappa^{\text{inc}}_{2,\mu}\colon H^2_{\text{inc}}(\mu)\to H^3_{\text{inc}}(\mu)$ that controls
  second-order lifts modulo isotriviality. If $\mu$ is smooth point in a reduced component and
  $\bigl(\kappa^{\text{inc}}_{2,\mu}\bigr)^{-1}(0)=\{0\}$, then the $G$-orbit of $\mu$ is Zariski
  open in that component. This provides a coordinate-free explanation of Richardson-type geometric
  rigidity even when the second deformation cohomology does not vanish.
\end{abstract}

\maketitle

\section*{Introduction}

Building on~\cite{Kaygun2025}, we give an intrinsic criterion for geometric rigidity in moduli
stacks of algebra laws. From a $\GL$-invariant quadratic presentation of algebraic laws, we
construct an incidence deformation complex and extract a canonical quadratic obstruction map
$\kappa^{\inc}_{2,\mu}$ governing second-order lifts of transverse first-order classes.  We show
that at a smooth point $\mu$ of a reduced irreducible component of the moduli scheme,
\emph{anisotropicity} of $\mu$ (i.e. $\bigl(\kappa^{\inc}_{2,\mu}\bigr)^{-1}(0)=\{0\}$) forces
Zariski openness of the $\GL$-orbit. This explains how geometric rigidity can persist even when
second deformation cohomology does not vanish.

Let $W$ be a finite-dimensional $\Bbbk$--vector space, set $W^{2,1}:=\Hom_\Bbbk(W^{\otimes 2},W)$,
and let $G:=\GL(W)$ act on $W^{2,1}$ by transport of structure. Following
Gabriel~\cite{Gabriel1972}, we consider a class of algebra laws cut out by a $G$--stable closed
subscheme $X\subseteq A_W$, where $A_W$ is the affine space of bilinear multiplications of the
chosen symmetry type (symmetric, skew-symmetric, or neither), and $X=\Var(Q)$ is defined by a
finite-dimensional $G$--stable subspace $Q\subseteq \Sym^2(A_W^\vee)$.  We write $[X/G]$ for the
corresponding quotient stack.\footnote{All constructions in the paper are $G$-equivariant on $X$,
  hence descend to $[X/G]$. In particular, an open orbit $G\cdot\mu\simeq G/H$ in $Z_{\red}$
  corresponds to an open substack $[G/H]$ of the induced component of $[X/G]$.} Then
$G(\Bbbk)$--orbits in $X(\Bbbk)$ parametrize isomorphism classes of algebra laws over $\Bbbk$.  A
basic problem is to detect when $\mu\in X(\Bbbk)$ has Zariski open orbit in its irreducible
component, i.e. when $\mu$ is \emph{geometrically rigid}.  The standard sufficient condition of
\emph{cohomological rigidity}, namely the vanishing of second deformation cohomology, is often too
strong: Richardson constructed families of geometrically rigid Lie algebras with
$H^2_{\Lie}(\mu)\neq 0$~\cite{Richardson1967}.

Polarization yields a $G$--equivariant bilinear evaluation map $\Theta\colon A_W\times A_W\to
Q^\vee$ and a canonical three-term complex on $X$, the \emph{incidence deformation
complex}~\eqref{eq:incidence-complex}.  Its fiber cohomology in degrees $2$ and $3$ encodes
first-order deformations modulo isotriviality and the primary obstruction space which agrees with
the usual low-degree deformation spaces on operadic loci~\cite{Kaygun2025}.  We then isolate the
second-order obstruction inherent in this complex: Proposition~\ref{prop:kappa-well-defined}
packages it into a canonical quadratic map $\kappa^{\inc}_{2,\mu}\colon H^2_{\inc}(\mu)\to
H^3_{\inc}(\mu)$, whose vanishing detects second-order liftability modulo isotriviality. Our main
result is that anisotropy forces geometric rigidity: In Theorem~\ref{thm:anisotropy-open-orbit} we
show that if $\mu\in (Z_{\red})_{\sm}$ is anisotropic on a reduced component $Z_{\red}$, then the
orbit $G\cdot\mu$ is Zariski open in $Z_{\red}$.

We illustrate this mechanism in the Lie case by constructing an anisotropic point with nontrivial
second cohomology: Proposition~\ref{prop:sl2-sym14-anisotropic} produces a Lie algebra $L$ with
$H^2_{\Lie}(L)\cong\mathbb C$ and nonzero quadratic obstruction.  Richardson's stability criterion
then confirms geometric rigidity~\cite{Richardson1967}.

Finally, we record two computable invariants on the open-orbit regime.  First, the $G$--equivariant
Gram morphism $\Gamma\colon A_W\to \Sym^2(W^\vee)$ induces a determinantal stratification by rank.
In Theorem~\ref{thm:dense-orbit-implies-generic-gram-rank} we show that on a component containing
an open orbit, the Gram rank equals the generic rank of the component on that orbit.  Next, in
Theorem~\ref{thm:chern-character-open-orbit}, we get a canonical Chern character dentity in
equivariant intersection theory expressed in the Chow ring in
$A_{G_\mu}^\ast(\mathrm{pt})\otimes\mathbb Q$, by restricting the incidence complex to an open
orbit $U_0\simeq G/G_\mu$.

\subsection*{What is known?}

The algebro--geometric approach to moduli of algebra laws on a fixed vector space $W$ goes back at
least to Gabriel \cite{Gabriel1972}: one studies the transport-of-structure action of $G=\GL(W)$ on
an affine parameter space of structure constants and analyzes orbit geometry inside the closed
subscheme cut out by the defining identities.  In very small dimensions this program is often
tractable and has produced explicit classifications together with degeneration pictures in several
operadic settings, notably for associative and Lie laws
\cite{Mazzola1979,GrunewaldOHalloran1988,GozeKhakimdjanov1996}, and a complete lists for
two--dimensional algebras \cite{KaygorodovVolkov2019}.  Comparable low-dimensional classifications
exist for other quadratic varieties of laws. For instance, nilpotent complex Leibniz algebras in
dimensions up to $4$ \cite{AlbeverioOmirovRakhimov2005}, $4$--dimensional Jordan algebras
\cite{Martin2012}, Novikov algebras in small dimensions \cite{BurdeDeGraaf2013}, and nilpotent Lie
algebras up to dimension $7$ \cite{Gong1998}.  From the incidence--variety perspective developed in
\cite{Kaygun2025}, such low-dimensional lists may be viewed as a family of concrete test cases in
which orbit closures, defect strata, and the tangent--obstruction package are simultaneously
computable and can be compared directly.  Beyond these regimes, however, complete classifications
are exceptional: Drozd’s tame--wild dichotomy suggests that the general classification problem is
typically of wild representation type \cite{Drozd1980}, so one should not expect a reasonable
global enumeration of orbits.

Locally at a point $\mu$ of the moduli scheme, the geometry is controlled by deformation
cohomology.  In the associative case, Gerstenhaber identified first-order deformations with
Hochschild $2$-cocycles and primary obstructions with Hochschild $3$-classes
\cite{Gerstenhaber1964}.  In the Lie case, Nijenhuis and Richardson gave the parallel description
via the Chevalley--Eilenberg complex and their graded Lie bracket, placing deformations and
obstructions into a Maurer--Cartan framework \cite{NijenhuisRichardson1967}.  Analogous deformation
theories exist for commutative and Leibniz structures, via Harrison cohomology \cite{Harrison1962}
and Leibniz cohomology \cite{LodayPirashvili1993}.

A persistent theme is that the vanishing of $H^2(\mu)$, while sufficient for smoothness and for
openness of the orbit in the classical Lie setting \cite{NijenhuisRichardson1967}, is not
necessary: Richardson’s examples show that geometric rigidity may occur even when $H^2(\mu)\neq 0$
\cite{Richardson1967}, so the obstruction to integrating first-order classes is genuinely
nonlinear.  Formal moduli methods explain this nonlinearity by encoding deformations in a
differential graded Lie algebra (or an $L_\infty$-model) and identifying obstructions with the
quadratic and higher terms of the Maurer--Cartan equation \cite{Schlessinger1968,
  GoldmanMillson1988}.  In concrete families one can sometimes compute these obstruction maps
explicitly; the work of Fialowski--Penkava and collaborators provides detailed miniversal
descriptions and obstruction calculi in low-dimensional Lie settings \cite{FialowskiPenkava2008,
  FialowskiPenkavaPhillipson2011, FialowskiPenkava2015}.

From the scheme of structure constants viewpoint, what is comparatively less standard is a direct,
intrinsic criterion forcing Zariski openness of $G\cdot\mu$ from the \emph{quadratic} part of the
obstruction theory, without choosing coordinates or fixing an ambient DGLA model.  The purpose of
this paper is to supply such a criterion for invariant quadratic presentations and to relate it to
explicit Richardson-type phenomena \cite{Richardson1967}.

\subsection*{Structure of the paper}

Section~\ref{sec:inc-var-fundamental} constructs, for any $G$--invariant quadratic presentation,
the fundamental incidence deformation complex and the associated incidence scheme, and records the
exact sequences controlling tangent data.

In Section~\ref{sec:coh-vs-geom} we study the rank stratifications induced by $\delta$ and $\Phi$
on an irreducible component $Z\subseteq X$; on the resulting dense open loci the cohomology sheaves
of $\mathcal C_X^\bullet$ are vector bundles. We relate cohomological and geometric rigidity by
showing that $H^2_{\inc}(\mu)=0$ implies that $G\!\cdot\!\mu$ is Zariski open in $Z_{\red}$, and we
record the converse at smooth points of $Z_{\red}$.

Section~\ref{sec:quadratic-anisotropy} defines the canonical quadratic obstruction
$\kappa^{\inc}_{2,\mu}\colon H^2_{\inc}(\mu)\to H^3_{\inc}(\mu)$ and introduces anisotropy
($\ker\kappa^{\inc}_{2,\mu}=0$). We prove that anisotropy at $\mu\in (Z_{\red})_{\sm}$ forces the
orbit $G\!\cdot\!\mu$ to be Zariski open, and we globalize $\kappa_2$ on constant-rank loci.

Section~\ref{sec:gram-stratification} introduces the $G$--equivariant Gram morphism
$\Gamma\colon A_W\to \Sym^2(W^\vee)$ and its rank stratification, which we use to distinguish
components admitting open orbits.

Finally, Section~\ref{sec:chern-open-orbits} restricts the incidence complex to a homogeneous open
orbit $U_0\simeq G/H$ and derives an equivariant Chern character identity constraining the
representation-theoretic data attached to an open orbit.

\subsection*{Notations and conventions}

Throughout, $\Bbbk$ is an algebraically closed field of characteristic $0$, all $\Bbbk$--vector
spaces are finite-dimensional, and $(\cdot)^\vee$ denotes $\Bbbk$--linear duals. Fix a
$\Bbbk$--vector space $W$ of dimension $m$ and put $G:=\GL(W)$, $\mathfrak g:=\Lie(G)\cong\End(W)$.

Let $\tau\colon W^{\otimes 2}\to W^{\otimes 2}$ be the flip. Since $\mathrm{char}(\Bbbk)=0$,
$W^{\otimes 2}\cong \Sym^2W \oplus \Lambda^2W$, and we set
\[
  W^{2,1}_+:=\Hom_\Bbbk(\Sym^2W,W),\quad
  W^{2,1}_-:=\Hom_\Bbbk(\Lambda^2W,W),\quad
  W^{2,1}:=\Hom_\Bbbk(W^{\otimes 2},W)=W^{2,1}_+\oplus W^{2,1}_-.
\]
In operadic examples we take $A_W:=W^{2,1}$ for associative and Leibniz laws, $A_W:=W^{2,1}_+$ for
commutative laws, and $A_W:=W^{2,1}_-$ for Lie laws, cf.~\cite{Kaygun2025}.

Fix a finite-dimensional $G$--stable subspace $Q\subseteq \Sym^2(A_W^\vee)$ and write
$X:=\Var(Q)\subseteq A_W$. We view $A_W$ as an affine space with underlying vector space $A_W$, so
$T_\mu A_W\simeq A_W$ for all $\mu\in A_W$. For $q\in Q$ we write $f_q(\nu):=q(\nu,\nu)$ for the
associated quadratic polynomial and use the polarized form
\[
  q(\nu_1,\nu_2)=\frac12\bigl(f_q(\nu_1+\nu_2)-f_q(\nu_1)-f_q(\nu_2)\bigr).
\]

For a $G$--scheme $Z$ and a $G$--representation $V$, the notation $Z\times V\to Z$ refers to the
trivial vector bundle with diagonal $G$--linearization.

We use Chow groups and operational Chow rings as in Fulton~\cite{Fulton1998}. For a finite type
$\Bbbk$--scheme $Y$, we write $A_\ast(Y)$ for cycles modulo rational equivalence and $A^\ast(Y)$ for
the operational Chow ring; if $Y$ is smooth of pure dimension $d$, we identify
$A^\ast(Y)\cong A_{d-\ast}(Y)$. Equivariantly, we write $A_G^\ast(Y)$ for the $G$--equivariant
operational Chow ring \cite{EdidinGraham1998}.

For a scheme $Y$ locally of finite type over $\Bbbk$, $Y_{\sm}\subseteq Y$ denotes the smooth locus
\cite[Tags 01V5, 02G1]{stacks-project}. For a coherent sheaf $\mathcal F$ on $Y$ and
$y\in Y(\Bbbk)$, we write $\mathcal F(y):=\mathcal F_y\otimes_{\mathcal O_{Y,y}}\Bbbk$ for the
fiber.

\section{The Fundamental Deformation Complex}\label{sec:inc-var-fundamental}

\subsection{The incidence scheme}\label{subsec:incidence}

Polarization yields the $G$--equivariant bilinear evaluation map
\begin{equation}\label{eq:Theta-evaluation}
  \Theta \colon A_W \times A_W \longrightarrow Q^\vee,
  \qquad
  \Theta(\mu,\nu)(q) := q(\mu,\nu).
\end{equation}
This induces a morphism of trivial vector bundles
$\Phi\colon \mathcal O_{A_W}\otimes A_W \to \mathcal O_{A_W}\otimes Q^\vee$ whose fiber at $\mu$ is
the linear map $\Phi_\mu(\nu)=\Theta(\mu,\nu)$. We also define a section $s_\Theta$ of the trivial
bundle $A_W\times A_W \times Q^\vee \to A_W\times A_W$ by
$s_\Theta(\mu,\nu) := (\mu, \nu, \Theta(\mu,\nu))$.

\begin{lemma}\label{lem:Inc-as-zerolocus}
  The incidence locus $\Inc(Q) \subset A_W \times A_W$ is defined as the scheme-theoretic zero
  locus $Z(s_\Theta)$ of the section $s_\Theta$ \cite[Tag 01M1]{stacks-project}. Its $\Bbbk$-points
  are precisely the pairs $(\mu,\nu)$ such that $q(\mu,\nu)=0$ for all $q\in Q$.
\end{lemma}

\begin{proof}
  Choose a basis $(q_1,\dots,q_r)$ of $Q$. The section $s_\Theta$ corresponds to the tuple of
  regular functions $(q_1(\mu,\nu),\dots,q_r(\mu,\nu))$. By definition, $Z(s_\Theta)$ is the closed
  subscheme cut out by the ideal generated by these functions. Thus, a point $(\mu,\nu)$ lies in
  $\Inc(Q)$ if and only if it is a common zero of these functions.
\end{proof}

The diagonal embedding $\Delta(\mu):=(\mu,\mu)$ pulls back $\Inc(Q)$ to the variety of algebras.

\begin{proposition}\label{prop:Var-as-pullback}
  The parameter scheme $X=\Var(Q)$ is defined as the scheme-theoretic inverse image of $\Inc(Q)$
  under $\Delta$. Explicitly, it is the closed subscheme of $A_W$ defined by the quadratic forms
  $\{f_q\}_{q\in Q}$.
\end{proposition}


\subsection{The fundamental exact sequence}\label{subsec:tangent-kernel}

Restrict $\Phi$ to $X$ to obtain a morphism of coherent sheaves
$\Phi_X \colon \mathcal O_X\otimes_\Bbbk A_W \to \mathcal O_X\otimes_\Bbbk Q^\vee$.  Define
$\mathcal S := \ker(\Phi_X)$ and $\mathcal N := \coker(\Phi_X)$.

\begin{lemma}\label{lem:tangent-space-kernel}
  Let $\mu\in X(\Bbbk)$. Under $T_\mu A_W\simeq A_W$, the Zariski tangent space $T_\mu X$ is
  identified with $\ker(\Phi_\mu)\subseteq A_W$.
\end{lemma}

\begin{proof}
  A tangent vector at $\mu$ is a $\Bbbk[\varepsilon]/(\varepsilon^2)$--point $\mu+\varepsilon\nu$
  satisfying $f_q(\mu+\varepsilon\nu)=0$ for all $q\in Q$. Expanding the quadratic form yields
  \[
    f_q(\mu+\varepsilon\nu) = f_q(\mu) + 2\varepsilon q(\mu,\nu) + \varepsilon^2 q(\nu,\nu).
  \]
  Since $\mu\in X$, $f_q(\mu)=0$. Modulo $\varepsilon^2$, the condition reduces to
  $2\varepsilon q(\mu,\nu)=0$. Thus $\Phi_\mu(\nu)=0$.
\end{proof}

\begin{proposition}\label{prop:tangent-sheaf-kernel}
  There is a canonical identification of coherent sheaves on $X$,
  $\mathcal T_X \cong \mathcal S=\ker(\Phi_X)$, fitting into the exact sequence:
  \begin{equation}\label{eq:fundamental-SES}
    0 \longrightarrow \mathcal S
    \longrightarrow \mathcal O_X\otimes_\Bbbk A_W
    \xrightarrow{ \Phi_X }
    \mathcal O_X\otimes_\Bbbk Q^\vee
    \longrightarrow \mathcal N
    \longrightarrow 0.
  \end{equation}
\end{proposition}

\begin{proof}
  Let $i\colon X\hookrightarrow A_W$ be the closed immersion with ideal sheaf $I$, generated by the
  quadrics $f_q(\mu)=q(\mu,\mu)$ for $q\in Q$. The standard incidence cokernel sequence for $i$
  gives an exact sequence of coherent sheaves on $X$
  \[
    I/I^2 \longrightarrow \Omega_{A_W}\big|_X \longrightarrow \Omega_X \longrightarrow 0,
  \]
  hence, after dualizing,
  \[
    0\longrightarrow \mathcal T_X
    \longrightarrow \mathcal T_{A_W}\big|_X
    \longrightarrow \mathcal Hom(I/I^2,\mathcal O_X).
  \]
  Using $\mathcal T_{A_W}\big|_X\simeq \mathcal O_X\otimes_\Bbbk A_W$, it remains to identify the
  rightmost arrow with $\Phi_X$ up to the harmless scalar~$2$ (invertible in $\Bbbk$).

  The generators $f_q$ induce a surjection $\mathcal O_X\otimes_\Bbbk Q\twoheadrightarrow I/I^2$,
  and hence an injection
  $\mathcal Hom(I/I^2,\mathcal O_X)\hookrightarrow \mathcal O_X\otimes_\Bbbk Q^\vee$. Under this
  injection, the composite map
  $\mathcal T_{A_W}\big|_X\to \mathcal Hom(I/I^2,\mathcal O_X)\hookrightarrow \mathcal
  O_X\otimes_\Bbbk Q^\vee$ is the Jacobian map $\nu\mapsto \bigl(q\mapsto df_q(\nu)\bigr)$. Since
  $df_q(\mu)(\nu)=2q(\mu,\nu)$, this composite is precisely $2\Phi_X$. Therefore
  \[
    \mathcal T_X
    = \ker\Bigl(\mathcal O_X\otimes_\Bbbk A_W \xrightarrow{ 2\Phi_X } \mathcal O_X\otimes_\Bbbk Q^\vee\Bigr)
    \cong \ker(\Phi_X)
     = \mathcal S.
  \]
  Finally, \eqref{eq:fundamental-SES} is the kernel--cokernel exact sequence of $\Phi_X$, with
  $\mathcal N:=\coker(\Phi_X)$.
\end{proof}

\subsection{The incidence deformation complex}\label{subsec:fundamental-complex}

For $\mu\in X(\Bbbk)$ and $\xi\in\mathfrak g=\Lie(G)$, define
$\delta_\mu(\xi) := \xi\cdot\mu \in T_\mu A_W \simeq A_W$. This globalizes to an
$\mathcal O_X$--linear morphism
$\delta\colon \mathcal O_X\otimes_\Bbbk\mathfrak g \to \mathcal O_X\otimes_\Bbbk A_W$.

\begin{lemma}\label{lem:Phi-delta-zero}
  One has $\Phi_X\circ\delta=0$. In particular, $\delta$ factors canonically through
  $\mathcal S:=\ker(\Phi_X)$.
\end{lemma}

\begin{proof}
  Since $Q$ is $G$-stable, the scheme $X$ is $G$-stable. Let $\mu \in X$. Then $f_q(g\cdot\mu)=0$
  for all $g\in G$. Differentiating the map $g \mapsto f_q(g\cdot\mu)$ at the identity yields
  $df_q(\mu)(\xi\cdot\mu)=0$. Since $df_q(\mu)(\nu) = 2q(\mu,\nu)$, this implies
  $\Phi_\mu(\delta_\mu(\xi))=0$.
\end{proof}

\begin{definition}\label{def:incidence-complex}
  The \emph{fundamental incidence deformation complex} on $X$ is the three-term complex of coherent
  $\mathcal O_X$--modules
  \begin{equation}\label{eq:incidence-complex}
    \mathcal C_X^\bullet\colon
    \mathcal C_X^1:=\mathcal O_X\otimes_\Bbbk\mathfrak g
    \xrightarrow{ \delta }
    \mathcal C_X^2:=\mathcal O_X\otimes_\Bbbk A_W
    \xrightarrow{ \Phi_X }
    \mathcal C_X^3:=\mathcal O_X\otimes_\Bbbk Q^\vee,
  \end{equation}
  concentrated in cohomological degrees $1,2,3$. Its cohomology sheaves are
  \[
    \mathcal H_X^1:=\ker(\delta),\qquad
    \mathcal H_X^2:=\ker(\Phi_X)/\im(\delta)=\mathcal S/\im(\delta),\qquad
    \mathcal H_X^3:=\coker(\Phi_X)=\mathcal N.
  \]
\end{definition}

\begin{definition}
  Fix $\mu\in X(\Bbbk)$. Taking the fiber of \eqref{eq:incidence-complex} at $\mu$ yields a complex
  of finite-dimensional $\Bbbk$--vector spaces
  \begin{equation}\label{eq:incidence-fiber-complex}
    \mathcal C_X^\bullet(\mu)\colon
    \mathfrak g \xrightarrow{\delta_\mu} A_W \xrightarrow{\Phi_\mu} Q^\vee, \quad \Phi_\mu\circ\delta_\mu=0,
  \end{equation}
  and we set $H^i_{\inc}(\mu):=H^i\bigl(\mathcal C_X^\bullet(\mu)\bigr)$ for $i=1,2,3$.  Although
  $H^i_{\inc}(\mu)$ is defined from the ambient $X$, geometric rigidity is a componentwise
  condition and depends on the comparison between $T_\mu Z_{\red}$ and $T_\mu X$.
\end{definition}

Note that $H^1_{\inc}(\mu)=\ker(\delta_\mu)=\Lie(G_\mu) \cong \Der(W,\mu)$. There are canonical
comparison maps $\mathcal H_X^i(\mu) \longrightarrow H^i_{\inc}(\mu)$, which are isomorphisms
whenever the ranks of the boundary maps are locally constant near $\mu$.

\begin{remark}\label{rem:operadic-loci}
  In the operadic loci $X_\Type=\Var(Q_\Type)$ for $\Type\in\{\Ass,\Comm,\Leib,\Lie\}$, the middle
  cohomology $H^2_{\inc}(\mu)$ recovers the usual first-order deformation spaces (Hochschild,
  Harrison, Leibniz, Chevalley--Eilenberg) as in \cite{Kaygun2025}, consistent with the spirit of
  Andr\'e--Quillen cohomology \cite{Quillen1968,Andre1974} and operadic tangent complexes
  \cite{HarpazNuitenPrasma2019}.
\end{remark}

\section{Cohomological vs. Geometric Rigidity}\label{sec:coh-vs-geom}

\subsection{Rank loci}\label{subsec:rank-loci}

We begin with the general determinantal picture: rank conditions for maps of finite locally free
modules define Zariski closed (or open) loci. One convenient formalism uses Fitting ideals.

\begin{lemma}\label{lem:determinantal-rank-loci}
  Let $Y$ be a scheme and let $\varphi\colon \mathcal E\to \mathcal F$ be a morphism of finite
  locally free $\mathcal O_Y$--modules of ranks $e$ and $f$. For each integer $r\ge 0$, the locus
  \[
    Y_{\le r}(\varphi) := \{y\in Y \mid \rank(\varphi\otimes\kappa(y))\le r\}
  \]
  is a closed subset of $Y$ (in fact a closed subscheme cut out by the $(r+1)\times(r+1)$ minors in
  any local trivialization). Consequently, $\rank(\varphi\otimes\kappa(y))$ is a lower
  semicontinuous function of $y$.
\end{lemma}

\begin{proof}
  The statement is Zariski local on $Y$. On an affine open $V=\Spec R$ over which
  $\mathcal E|_V\simeq \mathcal O_V^{\oplus e}$ and $\mathcal F|_V\simeq \mathcal O_V^{\oplus f}$,
  the map $\varphi|_V$ is given by an $f\times e$ matrix $A$ with entries in $R$. For a prime
  $\mathfrak p\subset R$, the rank of $A\otimes\kappa(\mathfrak p)$ is at most $r$ if and only if
  all $(r+1)\times(r+1)$ minors vanish in $\kappa(\mathfrak p)$, i.e. if and only if $\mathfrak p$
  contains the ideal generated by those minors. This ideal can be expressed as a suitable Fitting
  ideal of $\coker(\varphi|_V)$; see Stacks Project, Section~15.8 on Fitting ideals \cite[Tag
  07Z6]{stacks-project}. Hence $Y_{\le r}(\varphi)$ is closed on $V$, and these closed conditions
  glue.
\end{proof}

\begin{corollary}\label{cor:generic-rank}
  Let $Y$ be irreducible and let $\varphi\colon \mathcal E\to\mathcal F$ be as in
  Lemma~\ref{lem:determinantal-rank-loci}. Then there exists a dense open subset $U\subseteq Y$ on
  which $\rank(\varphi\otimes\kappa(y))$ is constant, equal to its maximum value on $Y$.
\end{corollary}

\begin{proof}
  Let $r_{\max}$ be the maximum of $\rank(\varphi\otimes\kappa(y))$ on $Y$. Then
  $Y_{\le r_{\max}-1}(\varphi)$ is closed by Lemma~\ref{lem:determinantal-rank-loci} and is a
  proper subset because $r_{\max}$ occurs somewhere. Its complement
  $U:=Y\setminus Y_{\le r_{\max}-1}(\varphi)$ is a dense open set on which the rank is
  $\ge r_{\max}$, hence equal to $r_{\max}$.
\end{proof}

We apply this twice, to the restrictions of $\Phi_X$ and $\delta$ to the reduced smooth locus of a
fixed component.

\subsection{Two rank stratifications}\label{subsec:two-rank-strata}

Fix an irreducible component $Z\subseteq X$, write $Z_{\red}$ for its reduced subscheme,
and set $U:=(Z_{\red})_{\sm}$. When we restrict the fundamental deformation complex
to $U$ we get a complex of vector bundles:
\[  \mathcal C_U^\bullet\colon
    \mathcal O_U\otimes_\Bbbk\mathfrak g
    \xrightarrow{ \delta_U }
    \mathcal O_U\otimes_\Bbbk A_W
    \xrightarrow{ \Phi_U }
    \mathcal O_U\otimes_\Bbbk Q^\vee.
\]
For integers $r,d\ge 0$ define the rank strata
\[
  U_{\Phi,=r}:=\{\mu\in U\mid \rank(\Phi_\mu)=r\},
  \qquad
  U_{\delta,=d}:=\{\mu\in U\mid \rank(\delta_\mu)=d\}.
\]
By Lemma~\ref{lem:determinantal-rank-loci}, the loci $\{\rank(\Phi_\mu)\le r\}$ and
$\{\rank(\delta_\mu)\le d\}$ are closed, hence each equality stratum is locally closed.

\begin{definition}\label{def:generic-ranks}
  Let $r_Z$ be the generic rank of $\Phi$ on $U$ and let $d_Z$ be the generic rank of $\delta$ on
  $U$, i.e. the constant values on dense opens provided by Corollary~\ref{cor:generic-rank}.
  Define the \emph{principal constant-rank locus}
  \[
    U^\circ := U_{\Phi,=r_Z}\cap U_{\delta,=d_Z} \subseteq U.
  \]
  By construction, $U^\circ$ is dense open in $U$.
\end{definition}

Two numerical consequences of these rank functions will be used repeatedly. First, since
$H^1_{\inc}(\mu)=\ker(\delta_\mu)$, rank--nullity gives
\begin{equation}\label{eq:H1-rankdelta}
  \dim_\Bbbk H^1_{\inc}(\mu)=\dim_\Bbbk(\mathfrak g)-\rank(\delta_\mu).
\end{equation}
Second, since $H^2_{\inc}(\mu)=\ker(\Phi_\mu)/\im(\delta_\mu)$, one has
\begin{equation}\label{eq:H2-ranks}
  \dim_\Bbbk H^2_{\inc}(\mu)=\dim_\Bbbk\ker(\Phi_\mu)-\rank(\delta_\mu).
\end{equation}
In particular, on $U^\circ$ the dimensions of $H^1_{\inc}(\mu)$ and
$H^2_{\inc}(\mu)$ are locally constant. Then Euler characteristic of the fiber of the
incidence complex \eqref{eq:incidence-complex} gives the numerical identity
\begin{equation}\label{eq:euler-identity}
  \dim_\Bbbk H^1_{\inc}(\mu)
  -\dim_\Bbbk H^2_{\inc}(\mu)
  +\dim_\Bbbk H^3_{\inc}(\mu)
  =
  \dim_\Bbbk(\mathfrak g)-\dim_\Bbbk(A_W)+\dim_\Bbbk(Q^\vee).
\end{equation}
Note that the left-hand side is a pointwise invariant of the fiber complex, whereas the right-hand
side is a uniform constant determined only by the ambient representation spaces.

\subsection{The constant-rank locus}\label{subsec:bundle-package}

On $U^\circ$ the morphisms $\Phi_U$ and $\delta_U$ have constant rank. This forces all the
algebraic objects appearing in the fundamental complex to be vector bundles, and the fiberwise
cohomology to be computed by the fibers of the cohomology sheaves.

\begin{lemma}\label{lem:kernel-cokernel-bundles}
  Let $Y$ be a scheme and let $\varphi\colon \mathcal E\to\mathcal F$ be a morphism of finite
  locally free $\mathcal O_Y$--modules. If $\rank(\varphi\otimes\kappa(y))$ is constant on an open
  subset $V\subseteq Y$, then $\ker(\varphi)|_V$ and $\coker(\varphi)|_V$ are finite locally free
  on $V$.
\end{lemma}

\begin{proof}
  The claim is local on $V$. Choose an affine open $\Spec R\subseteq V$ trivializing
  $\mathcal E,\mathcal F$, so that $\varphi$ is represented by a matrix $A\in M_{f\times e}(R)$.
  If the rank is constantly $r$ on $\Spec R$, then for every prime $\mathfrak p\subset R$ there is
  an $r\times r$ minor of $A$ not vanishing in $\kappa(\mathfrak p)$. Hence $\Spec R$ is covered by
  standard opens $D(\Delta)$ where some fixed $r\times r$ minor $\Delta$ becomes invertible. On
  $R_\Delta$, elementary row/column operations put $A$ into a block form with an $r\times r$
  identity block. In that form, $\ker(\varphi)$ and $\coker(\varphi)$ are visibly free of ranks
  $e-r$ and $f-r$, respectively. These local trivializations glue because they are defined on a
  cover by standard opens. A closely related statement for finite projective modules is proven in
  Stacks Project, Section~10.79 \cite[Lemma 10.79.4, Tag 05GD]{stacks-project}.
\end{proof}

\begin{proposition}\label{prop:bundle-package}
  On $U^\circ$ the sheaves $\mathcal S:=\ker(\Phi_X)$ and $\mathcal N:=\coker(\Phi_X)$ restrict to
  vector bundles, and the maps
  \[
    \delta|_{U^\circ}\colon \mathcal O_{U^\circ}\otimes_\Bbbk\mathfrak g \longrightarrow \mathcal S|_{U^\circ},
    \qquad
    \Phi|_{U^\circ}\colon \mathcal O_{U^\circ}\otimes_\Bbbk A_W \longrightarrow \mathcal O_{U^\circ}\otimes_\Bbbk Q^\vee
  \]
  have constant rank. In particular, the cohomology sheaves $\mathcal H_X^1=\ker(\delta)$ and
  $\mathcal H_X^2=\ker(\Phi_X)/\im(\delta)$ restrict to vector bundles on $U^\circ$.
\end{proposition}

\begin{proof}
  Apply Lemma~\ref{lem:kernel-cokernel-bundles} to $\Phi_U$ on $U_{\Phi,=r_Z}$ to see that
  $\mathcal S|_{U_{\Phi,=r_Z}}$ and $\mathcal N|_{U_{\Phi,=r_Z}}$ are vector bundles. Restricting
  further to $U^\circ$ preserves local freeness. On $U^\circ$ the map $\delta$ lands in
  $\mathcal S$ by Lemma~\ref{lem:Phi-delta-zero}, and has constant rank by definition of $U^\circ$.
  Another application of Lemma~\ref{lem:kernel-cokernel-bundles} to $\delta|_{U^\circ}$ (viewed as
  a morphism between vector bundles) yields that $\ker(\delta)|_{U^\circ}$ and
  $\coker(\delta|_{U^\circ})$ are vector bundles; the quotient
  $\mathcal H_X^2|_{U^\circ}=\mathcal S|_{U^\circ}/\im(\delta|_{U^\circ})$ is therefore locally
  free as well.
\end{proof}

\subsection{Rigidity implications}\label{subsec:rigidity-implications}

We now state the comparison between cohomological and geometric rigidity in a form that makes the
role of the two rank stratifications explicit.

\begin{definition}\label{def:two-rigidities}
  Let $\mu\in U(\Bbbk)$.  We call the point $\mu$ as \emph{cohomologically rigid} if
  $H^2_{\inc}(\mu)=0$, and as \emph{geometrically rigid in $Z_{\red}$} if
  $G\cdot\mu$ is Zariski open in $Z_{\red}$ (equivalently, if
  $\dim(G\cdot\mu)=\dim(Z_{\red})$ and $\mu$ is a smooth point of $Z_{\red}$).
\end{definition}

We use the following standard fact about locally closed subsets.

\begin{lemma}\label{lem:locally-closed-same-dim-open}
  Let $Y$ be an irreducible variety over $\Bbbk$ and let $W\subseteq Y$ be a locally closed
  subset. If $\dim W=\dim Y$, then $W$ is Zariski open in $Y$.
\end{lemma}

\begin{proof}
  Let $\overline{W}$ be the Zariski closure of $W$ in $Y$. Since $W$ is dense in $\overline{W}$ and
  $\dim W=\dim Y$, we have $\dim\overline{W}=\dim Y$. As $Y$ is irreducible, this forces
  $\overline{W}=Y$. Because $W$ is locally closed, it is open in $\overline{W}=Y$, hence open in
  $Y$.
\end{proof}

\begin{theorem}\label{thm:coh-rigid-implies-open-orbit}
  Let $\mu\in U^\circ(\Bbbk)$. If $H^2_{\inc}(\mu)=0$, then $G\cdot\mu$ is Zariski open in
  $Z_{\red}$.
\end{theorem}

\begin{proof}
  Since $Z_{\red}\subseteq X$ and $\mu\in Z_{\red}$, we have the inclusions of
  Zariski tangent spaces
  \[
    T_\mu(G\cdot\mu)\subseteq T_\mu Z_{\red} \subseteq T_\mu X.
  \]
  Under $T_\mu A_W\simeq A_W$, Lemma~\ref{lem:tangent-space-kernel} identifies $T_\mu X$ with
  $\ker(\Phi_\mu)$. On the other hand, by definition of $\delta_\mu$ as the differential of the
  orbit map, one has $T_\mu(G\cdot\mu)=\im(\delta_\mu)$; this is the usual identification of the
  tangent space to the orbit with the image of the infinitesimal action map (e.g. for a smooth
  algebraic group acting on a variety). Finally, on $U^\circ$ the condition $H^2_{\inc}(\mu)=0$ is
  equivalent to $\ker(\Phi_\mu)=\im(\delta_\mu)$, because
  $H^2_{\inc}(\mu)=\ker(\Phi_\mu)/\im(\delta_\mu)$ by definition.

  Putting these together yields
  \[
    \im(\delta_\mu)=T_\mu(G\cdot\mu)\subseteq T_\mu Z_{\red}\subseteq \ker(\Phi_\mu)=\im(\delta_\mu),
  \]
  hence $T_\mu Z_{\red}=T_\mu(G\cdot\mu)$ and therefore
  $\dim(Z_{\red})=\dim T_\mu Z_{\red}=\dim T_\mu(G\cdot\mu)=\dim(G\cdot\mu)$ since
  $\mu\in U=(Z_{\red})_{\sm}$. As $G\cdot\mu$ is locally closed for an algebraic group action,
  Lemma~\ref{lem:locally-closed-same-dim-open} shows that $G\cdot\mu$ is Zariski open in
  $Z_{\red}$.
\end{proof}

The converse direction is subtler because, in general, one only has an inclusion
$T_\mu Z_{\red}\subseteq T_\mu X=\ker(\Phi_\mu)$, and equality can fail when the quadratic
equations defining $X$ do not generate the radical ideal of $Z_{\red}$ to first order at $\mu$.

\section{Quadratic Obstructions and Anisotropy}\label{sec:quadratic-anisotropy}

\subsection{The quadratic obstruction}\label{subsec:anisotropy-second-order}

\begin{proposition}\label{prop:second-order-lift}
  Fix $\mu\in X(\Bbbk)$ and $\alpha,\beta\in A_W$. Set $\mu_t:=\mu+t\alpha+t^2\beta$.  Then
  $\mu_t\in X\bigl(\Bbbk[t]/(t^3)\bigr)$ if and only if $\Phi_\mu(\alpha)=0$ and
  $2 \Phi_\mu(\beta)+\Theta(\alpha,\alpha)=0$ in $Q^\vee$.
\end{proposition}

\begin{proof}
  For each $q\in Q$ we compute in $\Bbbk[t]/(t^3)$:
  \[
    f_q(\mu_t)=q(\mu_t,\mu_t)
    = q(\mu,\mu)+2t q(\mu,\alpha) + t^2\bigl(2q(\mu,\beta) + q(\alpha,\alpha)\bigr).
  \]
  Since $\mu\in X$, $q(\mu,\mu)=0$. Thus $f_q(\mu_t)=0$ in $\Bbbk[t]/(t^3)$ for all $q$ if and only
  if $q(\mu,\alpha)=0$ and $2q(\mu,\beta)+q(\alpha,\alpha)=0$ for all $q$, i.e.
  $\Phi_\mu(\alpha)=0$ and $2\Phi_\mu(\beta)+\Theta(\alpha,\alpha)=0$.
\end{proof}

Write $S_\mu:=\ker(\Phi_\mu)$ and $N_\mu:=\coker(\Phi_\mu)$ for the ``tangent'' and ``incidence
cokernel'' spaces at $\mu$ in the sense of Section~\ref{subsec:tangent-kernel}.
Proposition~\ref{prop:second-order-lift} shows that for a tangent direction $\alpha\in S_\mu$ the
obstruction to finding a second-order term $\beta$ lies in the class of $\Theta(\alpha,\alpha)$ in
$N_\mu$.

\subsection{Bilinear obstruction pairing}\label{subsec:anisotropy-kappa}

\begin{definition}\label{def:bilinear-obstruction}
  For $\mu\in X(\Bbbk)$ define a symmetric $\Bbbk$--bilinear pairing
  \[
    B_\mu\colon S_\mu\times S_\mu\longrightarrow N_\mu,
    \qquad
    B_\mu(\alpha,\beta):=\bigl[\Theta(\alpha,\beta)\bigr]\in \coker(\Phi_\mu).
  \]
  Its diagonal defines a homogeneous quadratic map
  \[
    \widetilde\kappa_{2,\mu}\colon S_\mu\longrightarrow N_\mu,
    \qquad
    \widetilde\kappa_{2,\mu}(\alpha):=B_\mu(\alpha,\alpha)=\bigl[\Theta(\alpha,\alpha)\bigr].
  \]
\end{definition}

\begin{lemma}\label{lem:cross-terms-exact}
  For $\mu\in X(\Bbbk)$, $\alpha\in S_\mu$, and $\xi\in\mathfrak g$, one has
  $\Theta\bigl(\alpha,\delta_\mu(\xi)\bigr)\in \im(\Phi_\mu)\subseteq Q^\vee$.
\end{lemma}

\begin{proof}
  By $G$--equivariance of $\Theta$ (Section~\ref{subsec:incidence}), differentiating at the
  identity gives, for all $\nu,\alpha\in A_W$ and $\xi\in\mathfrak g$,
  \[
    \xi\cdot\Theta(\nu,\alpha)=\Theta(\xi\cdot\nu,\alpha)+\Theta(\nu,\xi\cdot\alpha).
  \]
  Set $\nu=\mu$ and assume $\alpha\in S_\mu$, i.e. $\Theta(\mu,\alpha)=0$. Then
  \[
    0=\xi\cdot\Theta(\mu,\alpha)=\Theta(\xi\cdot\mu,\alpha)+\Theta(\mu,\xi\cdot\alpha)
    =
    \Theta\bigl(\delta_\mu(\xi),\alpha\bigr)+\Phi_\mu(\xi\cdot\alpha).
  \]
  Hence $\Theta(\delta_\mu(\xi),\alpha)\in\im(\Phi_\mu)$. Symmetry of $\Theta$ yields the claim.
\end{proof}

\begin{proposition}\label{prop:kappa-well-defined}
  Fix $\mu\in X(\Bbbk)$. The diagonal obstruction $\widetilde\kappa_{2,\mu}$ descends to a
  well-defined map
  \[
    \kappa^{\inc}_{2,\mu}\colon H^2_{\inc}(\mu)\longrightarrow H^3_{\inc}(\mu),
    \qquad
    \kappa^{\inc}_{2,\mu}([\alpha])=\bigl[\Theta(\alpha,\alpha)\bigr].
  \]
  Moreover, $\kappa^{\inc}_{2,\mu}([\alpha])=0$ if and only if the class
  $[\alpha]\in H^2_{\inc}(\mu)$ admits a second-order lift in the sense of
  Proposition~\ref{prop:second-order-lift}.
\end{proposition}

\begin{proof}
  We show invariance under changing representatives $\alpha\mapsto \alpha+\delta_\mu(\xi)$.  For
  $\alpha\in S_\mu$ and $\xi\in\mathfrak g$, expand using bilinearity:
  \[
    \Theta(\alpha+\delta_\mu(\xi),\alpha+\delta_\mu(\xi))
    =
    \Theta(\alpha,\alpha)
    +2\Theta\bigl(\alpha,\delta_\mu(\xi)\bigr)
    +\Theta\bigl(\delta_\mu(\xi),\delta_\mu(\xi)\bigr).
  \]
  By Lemma~\ref{lem:cross-terms-exact}, the cross term lies in $\im(\Phi_\mu)$. For the square
  term, note that Lemma~\ref{lem:Phi-delta-zero} gives $\delta_\mu(\xi)\in S_\mu$, hence applying
  Lemma~\ref{lem:cross-terms-exact} with $\alpha=\delta_\mu(\xi)$ shows
  $\Theta(\delta_\mu(\xi),\delta_\mu(\xi))\in\im(\Phi_\mu)$. Therefore
  $\Theta(\alpha+\delta_\mu(\xi),\alpha+\delta_\mu(\xi))\equiv\Theta(\alpha,\alpha)\pmod{\im(\Phi_\mu)}$,
  proving well-definedness on $S_\mu/\im(\delta_\mu)=H^2_{\inc}(\mu)$.

  For the lifting criterion, Proposition~\ref{prop:second-order-lift} shows that $\alpha\in S_\mu$
  admits $\beta$ with $\mu+t\alpha+t^2\beta\in X(\Bbbk[t]/(t^3))$ if and only if
  $\Theta(\alpha,\alpha)\in -2 \im(\Phi_\mu)$, equivalently $[\Theta(\alpha,\alpha)]=0$ in
  $\coker(\Phi_\mu)=H^3_{\inc}(\mu)$.
\end{proof}

\subsection{Anisotropy and open orbits}\label{subsec:anisotropy-open}

\begin{definition}\label{def:anisotropic}
  A point $\mu\in X(\Bbbk)$ is \emph{anisotropic} if $\kappa^{\inc}_{2,\mu}([\alpha])=0$ implies
  $[\alpha]=0$ in $H^2_{\inc}(\mu)$.
\end{definition}

\begin{theorem}\label{thm:anisotropy-open-orbit}
  Let $Z\subseteq X$ be the irreducible component containing $\mu$ and assume
  $\mu\in (Z_{\red})_{\sm}(\Bbbk)$. If $\mu$ is anisotropic, then $G\cdot\mu$ is Zariski open in
  $Z_{\red}$.
\end{theorem}

\begin{proof}
  Set $U:=Z_{\red}$. Since $G\cdot\mu\subseteq U$, we have
  \[
    T_\mu(G\cdot\mu)=\im(\delta_\mu)\subseteq T_\mu U \subseteq T_\mu X=\ker(\Phi_\mu)
  \]
  where the last inclusion is Lemma~\ref{lem:tangent-space-kernel}. Suppose $G\cdot\mu$ is not
  Zariski open in $U$. As $\mu$ is a smooth point of $U$, this implies
  $\dim T_\mu(G\cdot\mu)<\dim T_\mu U$, hence $\im(\delta_\mu)\subsetneq T_\mu U$. Choose
  $\alpha\in T_\mu U\setminus\im(\delta_\mu)$, so that $[\alpha]\neq 0$ in
  $\ker(\Phi_\mu)/\im(\delta_\mu)=H^2_{\inc}(\mu)$.
  
  The tangent vector $\alpha$ corresponds to a $\Bbbk[t]/(t^2)$--point of $U$ through $\mu$.
  Because $\mu$ is smooth on $U$ (hence $U\to\Spec(\Bbbk)$ is smooth at $\mu$), $U$ is formally
  smooth at $\mu$, so this $\Bbbk[t]/(t^2)$--point lifts to a $\Bbbk[t]/(t^3)$--point of $U$
  through $\mu$ \cite[Section~37.11]{stacks-project}. Equivalently, there exists $\beta\in A_W$
  such that $\mu_t=\mu+t\alpha+t^2\beta\in U(\Bbbk[t]/(t^3))\subseteq X(\Bbbk[t]/(t^3))$. By
  Proposition~\ref{prop:second-order-lift} we obtain $\kappa^{\inc}_{2,\mu}([\alpha])=0$,
  contradicting anisotropy. Therefore $\im(\delta_\mu)=T_\mu U$, hence $\dim(G\cdot\mu)=\dim U$,
  and since $G\cdot\mu$ is locally closed in $U$ it is Zariski open.
\end{proof}

\begin{remark}\label{rem:anisotropy-orbit-invariant}
  Anisotropy is constant along $G$--orbits: for $g\in G$, transport of structure identifies the
  complexes $\mathcal C_X^\bullet(\mu)$ and $\mathcal C_X^\bullet(g\cdot\mu)$, intertwining
  $\kappa^{\inc}_{2,\mu}$ and $\kappa^{\inc}_{2,g\cdot\mu}$. In particular, if $\mu$ is anisotropic
  then every point of $G\cdot\mu$ is anisotropic.
\end{remark}

\subsection{The anisotropic locus}\label{subsec:anisotropy-open-condition}

Anisotropy is defined pointwise by the condition $\ker(\kappa^{\inc}_{2,\mu})=0$, hence makes sense
whenever $\kappa^{\inc}_{2,\mu}$ is defined (in particular on $Z_{\sm}$ for a reduced component
$Z\subseteq X_{\red}$).  Since Theorem~\ref{thm:anisotropy-open-orbit} identifies anisotropy as a
sufficient criterion for geometric rigidity, the openness of the anisotropic locus can be deduced
formally from irreducibility.

\begin{proposition}\label{prop:anisotropy-open-locus}
  Let $Z\subseteq X_{\red}$ be an irreducible component. The anisotropic locus
  \[
    U_{\aniso}:=\{\mu\in Z_{\sm}\mid \ker(\kappa^{\inc}_{2,\mu})=0\}
  \]
  is Zariski open in $Z$ (hence in $Z_{\sm}$). Consequently, the anisotropic locus in $X_{\red}$ is
  Zariski open.
\end{proposition}

\begin{proof}
  By $G$--equivariance of $\kappa^{\inc}_{2,\mu}$, anisotropy is constant along $G$--orbits, hence
  $U_{\aniso}$ is a union of $G$--orbits. If $\mu\in U_{\aniso}$, then
  Theorem~\ref{thm:anisotropy-open-orbit} implies that $G\cdot\mu$ is Zariski open in $Z$. Since
  $Z$ is irreducible, it contains at most one Zariski open $G$--orbit, so $U_{\aniso}$ is either
  empty or equal to that open orbit, and in either case it is Zariski open. The final claim follows
  by taking the union over irreducible components of $X_{\red}$.
\end{proof}

\subsection{A cohomologically non-rigid anisotropic point}\label{subsec:anisotropic-example-sl2-sym14}

We record a classical example in which geometric rigidity (existence of an open $\GL(L)$--orbit in
the Lie variety) is already detected by the quadratic obstruction although the second cohomology
does not vanish.  This phenomenon goes back to Richardson~\cite{Richardson1967}.

Let $M:=\Sym^{14}(\mathbb C^2)$, $\dim M=15$, and set $L:=\fsl_2(\mathbb C)\ltimes M$, where $M$ is
an abelian ideal. Let $\mu\in\Hom(\Lambda^2L,L)$ denote the corresponding Lie bracket, so
$\mu\in X_{\Lie}(\mathbb C)$.  On the Lie operadic locus, the incidence complex identifies with the
Chevalley--Eilenberg deformation dg Lie algebra with adjoint coefficients (comparison in
\cite{Kaygun2025}); we write $H^i_{\Lie}(L):=H^i_{\mathrm{CE}}(L,L)$ and implicitly use the
identifications
\[
  H^2_{\inc}(\mu)\cong H^2_{\Lie}(L),\qquad
  H^3_{\inc}(\mu)\cong H^3_{\Lie}(L),\qquad
  \kappa^{\inc}_{2,\mu}([\alpha])=\Bigl[\frac12[\alpha,\alpha]_{\mathrm{NR}}\Bigr],
\]
where $[\ ,\ ]_{\mathrm{NR}}$ denotes the Nijenhuis--Richardson bracket on
$C^\bullet_{\mathrm{CE}}(L,L)$.

\begin{proposition}\label{prop:sl2-sym14-richardson-rigid-h2}
  The Lie algebra $L=\fsl_2(\mathbb C)\ltimes \Sym^{14}(\mathbb C^2)$ is geometrically rigid (its
  $\GL(L)$--orbit is open in the Lie variety) but $H^2_{\Lie}(L)\neq 0$.
\end{proposition}

\begin{proof}
  In~\cite[\S5]{Richardson1967}, Richardson fixes a $3$--dimensional simple Lie algebra
  $S\simeq\fsl_2(\mathbb C)$, lets $\rho$ be the irreducible representation of highest weight $2n$
  on $W\simeq \mathbb C^{2n+1}$, and defines $L_n:=S\ltimes_\rho W$.  By~\cite[\S5, Prop.~5.1,
  p.~344]{Richardson1967}, for every odd integer $n>5$ the Lie algebra $L_n$ is rigid and
  $H^2(L_n,L_n)\neq 0$.  For $n=7$ one has $W\simeq\Sym^{2n}(\mathbb C^2)=\Sym^{14}(\mathbb C^2)$,
  hence $L=L_7$, and the claim follows.
\end{proof}

We now isolate a concrete generator of $H^2_{\Lie}(L)$ and verify anisotropy by an explicit
Jacobiator computation.

Let $\Phi\in\Hom(\Lambda^2M,M)^{\fsl_2(\mathbb C)}$ be a nonzero alternating $\fsl_2$--equivariant
map, and let $\varphi\in\Hom(\Lambda^2L,L)$ be its extension by zero, i.e.
$\varphi|_{\Lambda^2M}=\Phi$ and $\varphi(x,-)=0$ if one argument lies in $\fsl_2(\mathbb C)$.

\begin{proposition}\label{prop:sl2-sym14-h2-dim1}
  One has $H^2_{\Lie}(L)\cong\mathbb C$. More precisely, $[\varphi]$ spans $H^2_{\Lie}(L)$, and
  $\Phi$ may be chosen to be the $7$--th transvectant $(\cdot,\cdot)_7$.
\end{proposition}

\begin{proof}
  The computation is carried out explicitly in~\cite[\S5]{Richardson1967} (see the argument on
  pp.~343--344 leading into Proposition~5.1), and we only record the structural input needed for
  later use.

  First, the Hochschild--Serre spectral sequence for $0\to M\to L\to\fsl_2(\mathbb C)\to 0$ with
  adjoint coefficients reduces $H^2_{\Lie}(L)$ to the $\fsl_2(\mathbb C)$--invariants in
  $H^2(M,L)$, because $H^1(\fsl_2,V)=H^2(\fsl_2,V)=0$ for all finite-dimensional $V$ (Whitehead)
  and $H^3(\fsl_2(\mathbb C),M)=0$ (since $M^{\fsl_2(\mathbb C)}=0$, and top-degree duality
  identifies $H^3(\fsl_2(\mathbb C),V)$ with the dual of the $\fsl_2$--coinvariants, hence vanishes
  when $V^{\fsl_2}=0$; see e.g.~\cite[\S1.5]{Fuks1986}). Thus
  \[
    H^2_{\Lie}(L)\;\cong\;H^2(M,L)^{\fsl_2(\mathbb C)}.
  \]

  Next, since $M$ is abelian and acts on $L$ via $\ad|_M$, one has
  $B^2(M,L)\subseteq\Hom(\Lambda^2M,M)$ and, taking invariants, $B^2(M,L)^{\fsl_2(\mathbb C)}=0$
  because $\Hom(M,\fsl_2(\mathbb C))^{\fsl_2(\mathbb C)}=0$ (Clebsch--Gordan; see
  \cite[\S11.2]{FultonHarris1991}). Since $\fsl_2(\mathbb C)$ is semisimple in characteristic $0$,
  the invariants functor is exact on finite-dimensional modules, so
  $H^2(M,L)^{\fsl_2(\mathbb C)}\cong Z^2(M,L)^{\fsl_2(\mathbb C)}$.

  Finally, classical invariant theory gives
  \[
    \dim \Hom(\Lambda^2M,M)^{\fsl_2(\mathbb C)}=1,\qquad
    \dim \Hom(\Lambda^2M,\fsl_2(\mathbb C))^{\fsl_2(\mathbb C)}=1,
  \]
  generated respectively by the $7$--th transvectant $(\cdot,\cdot)_7$ and the $13$--th
  transvectant $\psi$~\cite[\S2--3]{Chipalkatti2006}. Richardson shows that $\psi$ is not a cocycle
  (equivalently $d\psi\neq 0$), so the $\fsl_2$--invariant cocycles in bidegree $(0,2)$ are
  precisely the $M$--valued ones, and hence $H^2_{\Lie}(L)\cong\mathbb C$ generated by the class of
  $\Phi:=(\cdot,\cdot)_7$ (and its extension-by-zero $\varphi$).

  Moreover, $\varphi$ is a $2$--cocycle: since $[M,M]=0$, one has $(d\varphi)|_{\Lambda^3M}=0$, and
  for $x\in\fsl_2(\mathbb C)$ and $u,v\in M$,
  \[
    (d\varphi)(x,u,v)=x\cdot\Phi(u,v)-\Phi(x\cdot u,v)-\Phi(u,x\cdot v)=0
  \]
  by $\fsl_2(\mathbb C)$--equivariance of $\Phi$.
\end{proof}

\begin{proposition}\label{prop:sl2-sym14-anisotropic}
  One has $\kappa^{\inc}_{2,\mu}([\varphi])\neq 0\in H^3_{\Lie}(L)$. Hence
  $\ker(\kappa^{\inc}_{2,\mu})=0$, so $\mu$ is anisotropic while $H^2_{\inc}(\mu)\neq 0$.
\end{proposition}

\begin{proof}
  Under the Lie-locus identification, the obstruction of $[\varphi]$ is represented by
  $\frac12[\varphi,\varphi]_{\mathrm{NR}}$. Since $\varphi$ is supported on $\Lambda^2M$ with values
  in $M$, the restriction of $\frac12[\varphi,\varphi]_{\mathrm{NR}}$ to $\Lambda^3M$ is the
  Jacobiator of $\Phi$,
  \[
    J_\Phi(u,v,w):=\Phi(\Phi(u,v),w)+\Phi(\Phi(v,w),u)+\Phi(\Phi(w,u),v) \in
    \Hom(\Lambda^3M,M)^{\fsl_2(\mathbb C)}.
  \]
  As recalled above, $\Hom(\Lambda^2M,\fsl_2(\mathbb C))^{\fsl_2(\mathbb C)}=\mathbb C\cdot\psi$ is
  one-dimensional and $d\psi\neq 0$ (Richardson, loc.\ cit.), so the $\fsl_2$--invariant coboundary
  subspace in $\Hom(\Lambda^3M,M)^{\fsl_2(\mathbb C)}$ is exactly $\mathbb C\cdot d\psi$.

  To see that $J_\Phi\notin\mathbb C\cdot d\psi$, we evaluate on two weight-homogeneous test
  triples. Let $v_i=x^{14-i}y^i$ be the standard monomial basis of $M=\Sym^{14}(\mathbb C^2)$, and
  take $\Phi=(\cdot,\cdot)_7$ and $\psi=(\cdot,\cdot)_{13}$ with the unnormalized transvectant
  convention of~\cite{Chipalkatti2006}. A direct transvectant computation gives
  \[
    \frac{J_\Phi(v_0,v_1,v_{13})}{d\psi(v_0,v_1,v_{13})}
    =\frac{24024}{5}
    \qquad\neq\qquad
    \frac{J_\Phi(v_1,v_2,v_{14})}{d\psi(v_1,v_2,v_{14})}
    =-7392.
  \]
  Each numerator and denominator above is a nonzero scalar multiple of a single basis vector
  (respectively $v_0$ and $v_3$) by weight considerations, so the ratios are well-defined.  Since
  the ratios disagree, $J_\Phi$ is not a scalar multiple of $d\psi$, hence $[J_\Phi]\neq 0$ in
  $H^3_{\Lie}(L)$ and $\kappa^{\inc}_{2,\mu}([\varphi])=[J_\Phi]\neq 0$.

  Since $H^2_{\Lie}(L)\cong\mathbb C$ by Proposition~\ref{prop:sl2-sym14-h2-dim1}, the kernel of
  $\kappa^{\inc}_{2,\mu}$ is trivial.
\end{proof}

\begin{remark}\label{rem:sl2-sym14-richardson}
  Proposition~\ref{prop:sl2-sym14-richardson-rigid-h2} is precisely the “rigid but $H^2\neq 0$”
  phenomenon: the $\GL(L)$--orbit of $\mu$ is open (geometric rigidity), yet
  $H^2_{\inc}(\mu)\cong H^2_{\Lie}(L)\neq 0$ (cohomological non-rigidity).
  Proposition~\ref{prop:sl2-sym14-anisotropic} shows that, in this example, openness is already
  detected by anisotropy of the quadratic obstruction.
\end{remark}

\section{Gram Stratification}\label{sec:gram-stratification}

The purpose of this section is to introduce a $G$--equivariant numerical invariant which is
constant on $G$--orbits and therefore can be used to distinguish the open orbits produced in
Section~\ref{sec:coh-vs-geom}. The invariant is the rank of a trace--type symmetric bilinear form.

\subsection{The Gram morphism}\label{subsec:gram-morphism}

\begin{definition}\label{def:gram-morphism}
  For $\mu\in A_W$ and $v\in W$, define the right multiplication operator $R_v^\mu\in\End(W)$ by
  $R_v^\mu(w):=\mu(w,v)$. The \emph{Gram form} associated to $\mu$ is the symmetric bilinear form
  $\gamma_\mu\in\Sym^2(W^\vee)$ defined by
  \[
    \gamma_\mu(v,w) := \Tr\bigl(R_v^\mu R_w^\mu\bigr).
  \]
  The \emph{Gram morphism} $\Gamma(\mu):=\gamma_\mu$ is a map of affine schemes of the form
  $\Gamma\colon A_W \longrightarrow \Sym^2(W^\vee)$.
\end{definition}

Symmetry is immediate from cyclicity of trace:
$\gamma_\mu(v,w)=\Tr(R_v^\mu R_w^\mu)=\Tr(R_w^\mu R_v^\mu)=\gamma_\mu(w,v)$.  That $\Gamma$ is a
morphism is also straightforward: after choosing a basis of $W$, the matrix entries of $R_v^\mu$
depend $\Bbbk$--linearly on the structure constants of $\mu$, hence the entries of
$R_v^\mu R_w^\mu$ are polynomial (indeed quadratic) in those constants, and the trace is polynomial
in the entries.

The form $\gamma_\mu$ determines a linear map
\[
  \gamma_\mu^\sharp \colon W \longrightarrow W^\vee,\qquad
  v\longmapsto \gamma_\mu(v,-),
\]
whose rank equals $\rank(\gamma_\mu)$ in the usual sense. Globalizing, $\Gamma$ induces a morphism
of trivial vector bundles on $A_W$,
\begin{equation}\label{eq:Gamma-sharp}
  \Gamma^\sharp \colon \mathcal O_{A_W}\otimes_\Bbbk W \longrightarrow \mathcal O_{A_W}\otimes_\Bbbk W^\vee,
\end{equation}
whose fiber at $\mu$ is precisely $\gamma_\mu^\sharp$.

\begin{lemma}\label{lem:Gamma-equivariant}
  The Gram morphism $\Gamma$ is $G$--equivariant. Equivalently, for all $g\in G$ and $\mu\in A_W$
  one has
  \[
    \gamma_{g\cdot\mu}(gv,gw)=\gamma_\mu(v,w)\qquad (v,w\in W),
  \]
  and in particular $\rank(\gamma_{g\cdot\mu})=\rank(\gamma_\mu)$.
\end{lemma}

\begin{proof}
  Writing $(g\cdot\mu)(x,y)=g \mu(g^{-1}x,g^{-1}y)$, one checks directly that
  \[
    R_{gv}^{ g\cdot\mu}=g R_v^\mu g^{-1}.
  \]
  Hence
  \[
    \gamma_{g\cdot\mu}(gv,gw)
    =\Tr\bigl(R_{gv}^{ g\cdot\mu} R_{gw}^{ g\cdot\mu}\bigr)
    =\Tr\bigl(g R_v^\mu R_w^\mu g^{-1}\bigr)
    =\Tr\bigl(R_v^\mu R_w^\mu\bigr)
    =\gamma_\mu(v,w),
  \]
  and rank is invariant under change of basis.
\end{proof}

\subsection{Rank strata and the generic Gram rank}\label{subsec:gram-rank-strata}

Restrict $\Gamma^\sharp$ to $X=\Var(Q)\subseteq A_W$ to obtain a morphism of vector bundles on $X$,
\[
  \Gamma_X^\sharp\colon \mathcal O_X\otimes_\Bbbk W \longrightarrow \mathcal O_X\otimes_\Bbbk W^\vee,
\]
whose fiber at $\mu\in X(\Bbbk)$ has rank $\rank(\gamma_\mu)$.

\begin{definition}\label{def:gram-rank-strata}
  For $r\ge 0$ define the \emph{Gram rank loci} on $X$ by
  \[
    X^\Gamma_{\le r}:=\{\mu\in X \mid \rank(\gamma_\mu)\le r\},
    \qquad
    X^\Gamma_{=r}:=X^\Gamma_{\le r}\setminus X^\Gamma_{\le r-1}.
  \]
\end{definition}

\begin{lemma}\label{lem:gram-rank-loci-closed}
  For each $r\ge 0$, the subset $X^\Gamma_{\le r}$ is Zariski closed in $X$ and is $G$--stable.
  Consequently, $X^\Gamma_{=r}$ is locally closed and $G$--stable.
\end{lemma}

\begin{proof}
  Closedness is determinantal: $X^\Gamma_{\le r}$ is the locus where the fiber rank of
  $\Gamma_X^\sharp$ is at most $r$, hence is cut out by the vanishing of all $(r+1)\times(r+1)$
  minors in a local trivialization of \eqref{eq:Gamma-sharp}. This is exactly the general
  determinantal rank construction for morphisms of vector bundles (cf.\
  Lemma~\ref{lem:determinantal-rank-loci}). $G$--stability follows from
  Lemma~\ref{lem:Gamma-equivariant}.
\end{proof}

\begin{proposition}\label{prop:generic-gram-rank}
  Let $Z\subseteq X$ be an irreducible component, and let $Z_{\red}$ be its reduced
  subscheme. There exists a unique integer $\rho(Z)$, the \emph{generic Gram rank}, such that the
  locus
  \[
    U_\Gamma(Z) := Z_{\red}\cap X^\Gamma_{=\rho(Z)}
  \]
  is Zariski dense and open in $Z_{\red}$.
\end{proposition}

\begin{proof}
  By Lemma~\ref{lem:gram-rank-loci-closed}, the rank function
  $\mu\mapsto \rank(\gamma_\mu)=\rank\bigl((\Gamma_X^\sharp)_\mu\bigr)$ is lower semicontinuous on
  $Z_{\red}$. Hence it attains a maximum value on $Z_{\red}$, and the locus where this maximum is
  attained is dense open in the irreducible variety $Z_{\red}$
  (cf. Corollary~\ref{cor:generic-rank}). Denote this maximum by $\rho(Z)$; then
  $U_\Gamma(Z)=Z_{\red}\cap X^\Gamma_{=\rho(Z)}$ is precisely that dense open locus, and uniqueness
  of $\rho(Z)$ is immediate.
\end{proof}

\subsection{Dense orbits lie in the generic Gram stratum}\label{subsec:gram-open-orbits}

The Gram rank is constant on $G$--orbits, so any orbit which is dense in an irreducible component
must lie in the component's generic Gram stratum.

\begin{theorem}\label{thm:dense-orbit-implies-generic-gram-rank}
  Let $Z\subseteq X$ be an irreducible component, and let $\rho(Z)$ be its generic Gram rank from
  Proposition~\ref{prop:generic-gram-rank}. Let $\mu\in Z_{\red}(\Bbbk)$ and assume that
  $\overline{G\cdot\mu}=Z_{\red}$ (in particular, this holds if $G\cdot\mu$ is Zariski open in
  $Z_{\red}$). Then $\rank(\gamma_\mu)=\rho(Z)$.
\end{theorem}

\begin{proof}
  Set $r:=\rank(\gamma_\mu)$. By Lemma~\ref{lem:Gamma-equivariant}, $\rank(\gamma)$ is constant on
  $G\cdot\mu$, hence $G\cdot\mu\subseteq X^\Gamma_{\le r}$. The latter is closed by
  Lemma~\ref{lem:gram-rank-loci-closed}, so it contains the closure $\overline{G\cdot\mu}$ in $X$.
  By hypothesis $\overline{G\cdot\mu}=Z_{\red}$, hence $Z_{\red}\subseteq X^\Gamma_{\le
    r}$. Therefore every point of $Z_{\red}$ has Gram rank at most $r$, so the maximal
  (equivalently generic) Gram rank on $Z_{\red}$ satisfies $\rho(Z)\le r$.

  On the other hand, $\mu\in Z_{\red}$ has Gram rank $r$, so $\rho(Z)$, being the maximum
  Gram rank on $Z_{\red}$, satisfies $r\le \rho(Z)$. Thus $r=\rho(Z)$.
\end{proof}

\begin{remark}\label{rem:gram-organizes-components}
  Theorem~\ref{thm:dense-orbit-implies-generic-gram-rank} shows that whenever an irreducible
  component $Z$ supports a dense (in particular open) $G$--orbit, the Gram rank of any point on
  that orbit is forced to equal the component invariant $\rho(Z)$. Hence, to distinguish open
  orbits living on different components, it suffices to compute $\rank(\gamma_\mu)$ at one
  representative $\mu$ on each open orbit and compare the resulting values of $\rho(Z)$.
\end{remark}

\section{Equivariant Chern Characters on Open Orbits}\label{sec:chern-open-orbits}

Fix an irreducible component $Z\subseteq X$ and assume that $Z_{\red}$ contains a Zariski open
$G$--orbit
\[
  U_0:=G\cdot\mu \subseteq Z_{\red},\qquad U_0\simeq G/H,\qquad H:=G_\mu,\qquad \mathfrak h:=\Lie(H).
\]
Note that if $H$ is reductive, then $U_0\simeq G/H$ is affine by Matsushima's criterion
\cite{Matsushima1960,Arzhantsev2008}; in particular, $U_0$ is an affine open subscheme of the
affine scheme $Z_{\red}$.

The numerical Euler identity \eqref{eq:euler-identity} is the degree--$0$ shadow of a canonical
identity in equivariant intersection theory. The point is that on $U_0$ the three terms of the
incidence complex become \emph{associated bundles} on the homogeneous space $G/H$, and one can
therefore compute their equivariant Chern characters inside
$A_H^\ast(\mathrm{pt})\otimes\mathbb Q$.  This description interacts naturally with the Gram
stratification: on $U_0$ the Gram form $\gamma_\mu$ has constant rank and yields a stabilizer
invariant subrepresentation $\rad(\gamma_\mu)\subseteq W$, which can be compared with the
representation-theoretic constraints forced by the Chern character identity.

\subsection{Equivariant Chow and associated bundles on $G/H$}\label{subsec:chern-homogeneous}

Equivariant intersection theory identifies the equivariant Chow ring of a homogeneous space with
the equivariant Chow ring of a point for the stabilizer:
\begin{equation}\label{eq:EG-homogeneous}
  A_G^\ast(G/H)\cong A_H^\ast(\mathrm{pt}),
\end{equation}
compatibly with equivariant Chern classes and pullback along $G/H\to\mathrm{pt}$; see
\cite[\S\S2--3]{EdidinGraham1998}. Moreover, every $G$--equivariant vector bundle on $G/H$ is an
associated bundle $G\times^H V$ for a finite-dimensional $H$--module $V$, and the correspondence
$V\mapsto G\times^H V$ identifies $K_0^G(G/H)$ with $K_0^H(\mathrm{pt})$.  Equivariant
Riemann--Roch provides a Chern character
\[
  \ch_G\colon K_0^G(G/H)\longrightarrow A_G^\ast(G/H)\otimes\mathbb Q,
\]
which under \eqref{eq:EG-homogeneous} becomes the usual equivariant Chern character
$\ch_H\colon K_0^H(\mathrm{pt})\to A_H^\ast(\mathrm{pt})\otimes\mathbb Q$; see
\cite{EdidinGraham1998,Thomason1992}.

\subsection{The incidence complex on an open orbit}\label{subsec:chern-incidence}

Recall the incidence complex \eqref{eq:incidence-complex} of $G$--equivariant coherent sheaves on
$X$ with cohomology sheaves $\mathcal H_X^i$. Since the three terms are trivial $G$--bundles, their
restriction to $U_0\simeq G/H$ is determined by the underlying $H$--modules $\mathfrak g$, $A_W$,
and $Q^\vee$.

\begin{lemma}\label{lem:associated-incidence-on-orbit}
  The restriction of $\mathcal C_X^\bullet$ to $U_0\simeq G/H$ is canonically isomorphic to the
  complex of associated bundles induced from the $H$--module
  complex~\eqref{eq:incidence-fiber-complex}.  Consequently, for $i=1,2,3$ one has
  $\mathcal H_X^i\big|_{U_0}\cong G\times^H H^i_{\inc}(\mu)$.
\end{lemma}

\begin{proof}
  The differentials $\delta$ and $\Phi_X$ are $G$--equivariant morphisms of $G$--bundles on $X$,
  hence on $U_0$ they are determined by their fibers at the basepoint $\mu$; these fibers are
  $\delta_\mu$ and $\Phi_\mu$. The first claim follows by descent from $G$ to $H$, and the
  identification of cohomology sheaves follows from exactness of the associated bundle functor
  $G\times^H(-)$ on finite-dimensional $H$--modules.
\end{proof}

\begin{theorem}\label{thm:chern-character-open-orbit}
  Under the identification $A_G^\ast(G/H)\cong A_H^\ast(\mathrm{pt})$, the incidence cokernel
  $\mathcal N(\mu)=\coker(\Phi_\mu)$ satisfies in $A_H^\ast(\mathrm{pt})\otimes\mathbb Q$ the
  identity
  \begin{equation}\label{eq:ch-identity}
    \ch_H\bigl(\mathcal N(\mu)\bigr)
    = \ch_H(Q^\vee)-\ch_H(A_W)+\ch_H(\mathfrak g)-\ch_H(\mathfrak h)
      + \ch_H\bigl(H^2_{\inc}(\mu)\bigr).
  \end{equation}
\end{theorem}

\begin{proof}
  In $K_0^G(U_0)$ one has the standard Euler relation for bounded complexes of $G$--vector bundles,
  \[
    [\mathcal C_X^1|_{U_0}]-[\mathcal C_X^2|_{U_0}]+[\mathcal C_X^3|_{U_0}]
    = [\mathcal H_X^1|_{U_0}]-[\mathcal H_X^2|_{U_0}]+[\mathcal H_X^3|_{U_0}],
  \]
  cf. \cite[\S IV.2]{SGA6} or \cite[Ch.~IV]{WeibelKBook}, and in the equivariant setting
  \cite{Thomason1987}. Transport this equality along the equivalence
  $\mathrm{Vect}_G(G/H)\simeq\Rep(H)$ and use Lemma~\ref{lem:associated-incidence-on-orbit}. The
  terms correspond to the $H$--modules $\mathfrak g$, $A_W$, $Q^\vee$,
  $\mathfrak h=\ker(\delta_\mu)$, $H^2_{\inc}(\mu)$, and $\mathcal
  N(\mu)=\coker(\Phi_\mu)$. Applying $\ch_H$ yields \eqref{eq:ch-identity}.
\end{proof}

\begin{remark}\label{rem:chern-degree-zero}
  Taking the degree--$0$ component of \eqref{eq:ch-identity} recovers the numerical Euler identity
  \eqref{eq:euler-identity}. Thus \eqref{eq:ch-identity} should be viewed as the intrinsic
  refinement of \eqref{eq:euler-identity} on open orbits.
\end{remark}

\subsection{Weight calculus}\label{subsec:chern-weights}

Choose a maximal torus $T\subseteq H$ with Weyl group $W_H$. After tensoring with $\mathbb Q$ one
has
\begin{equation}\label{eq:Weyl-invariants}
  A_H^\ast(\mathrm{pt})\otimes\mathbb Q \cong \bigl(A_T^\ast(\mathrm{pt})\otimes\mathbb Q\bigr)^{W_H},
\end{equation}
see \cite[\S3]{EdidinGraham1998}, and $A_T^\ast(\mathrm{pt})\cong\Sym(X^\ast(T))$. If an
$H$--module $V$ decomposes into $T$--weights $\lambda_1,\dots,\lambda_r$ (with multiplicities),
then
\begin{equation}\label{eq:ch-weights}
  \ch_T(V)=\sum_{i=1}^r e^{\lambda_i}\in A_T^\ast(\mathrm{pt})\otimes\mathbb Q,
\end{equation}
where $e^{\lambda}$ denotes the graded exponential series. Consequently, the identity
\eqref{eq:ch-identity} reduces to the determination of the $H$--module structure of
\[
  \mathfrak h,\qquad \mathcal N(\mu),\qquad H^2_{\inc}(\mu),
\]
since the remaining terms $A_W$ and $Q^\vee$ are functorial Schur constructions on $W$ in the
operadic loci, and $\mathfrak g\simeq W\otimes W^\vee$.

\subsection{Gram rank on an open orbit}\label{subsec:gram-vs-chern}

Assume $U_0=G\cdot\mu$ is open in $Z_{\red}$. By $G$--equivariance of the Gram morphism $\Gamma$
(Section~\ref{sec:gram-stratification}), the Gram rank is constant on $U_0$ and equals the generic
Gram rank of $Z$; we denote it by $\rho_0(Z)$:
\begin{equation}\label{eq:gram-generic-open-orbit}
  \rho_0(Z)=\rank(\gamma_\mu),\qquad \rad(\gamma_\mu)\subseteq W \text{is an $H$--submodule of dimension }\dim W-\rho_0(Z).
\end{equation}
Thus, on an open orbit, one simultaneously has the representation-theoretic constraint
\eqref{eq:ch-identity} in $A_H^\ast(\mathrm{pt})\otimes\mathbb Q$ and the additional stabilizer
invariant $\rad(\gamma_\mu)\subseteq W$. In practice, $\rho_0(Z)$ is often coarse but robust: it
immediately rules out the possibility that two open orbits lie in the same component when their
generic Gram ranks differ, while \eqref{eq:ch-identity} refines this by constraining the full
$H$--character of $\mathcal N(\mu)$ via the correction term $\ch_H(H^2_{\inc}(\mu))$.

The interaction between \eqref{eq:ch-identity} and \eqref{eq:gram-generic-open-orbit} becomes
particularly transparent on standard operadic loci, where the Gram form specializes to a familiar
trace form and its radical is controlled by standard structural ideals. We record the resulting
uniform pattern in the next subsection.

\subsection{Operadic examples}\label{subsec:chern-examples}

\subsubsection*{Lie locus}
Assume $X=X_{\Lie}\subseteq A_W\cong W\otimes\Lambda^2W^\vee$ and
$(Q_{\Lie})^\vee\cong W\otimes\Lambda^3W^\vee$ encodes the polarized Jacobi identity. If the
$T$--weights of $W$ are $\chi_1,\dots,\chi_m$, then $A_W$ and $(Q_{\Lie})^\vee$ have $T$--weights
\[
  \mathrm{wt}_T(A_W)=\{\chi_i-\chi_j-\chi_k\mid 1\le i\le m, 1\le j<k\le m\},
\]
\[
  \mathrm{wt}_T\bigl((Q_{\Lie})^\vee\bigr)=\{\chi_i-\chi_j-\chi_k-\chi_\ell\mid 1\le i\le m, 1\le j<k<\ell\le m\},
\]
while $\mathfrak g\simeq W\otimes W^\vee$ contributes weights $\chi_i-\chi_j$. Hence, once the
$H$--modules $\mathfrak h$, $\mathcal N(\mu)$, and $H^2_{\inc}(\mu)\cong H^2_{\Lie}(W,\mu)$ are
known, the identity \eqref{eq:ch-identity} becomes an explicit equality in
$\Sym(X^\ast(T))^{W_H}\otimes\mathbb Q$.

On the other hand, the Gram form is the Killing form $\gamma_\mu(v,w)=\Tr(\ad_v\ad_w)$. Thus
$\rho_0(Z)=\dim W$ on any component whose open orbit consists of semisimple Lie algebras, by
Cartan's criterion \cite[\S III.5]{Humphreys1972}. For semidirect products $L=\fsl_2\ltimes M$ with
$M$ abelian one has $\ad_m(L)\subseteq M$ and $\ad_m(M)=0$, hence $M\subseteq\rad(\gamma_\mu)$ and
$\rho_0(Z)\le \dim W-\dim M$. In particular, in the example $L=\fsl_2\ltimes\Sym^{14}(\mathbb C^2)$
from Section~\ref{subsec:anisotropic-example-sl2-sym14}, the open orbit
(Richardson~\cite{Richardson1967}) lies in a component whose generic Gram rank is strictly less
than $\dim W$, while \eqref{eq:ch-identity} records the nontrivial correction term coming from
$H^2_{\Lie}(L)$.

\subsubsection*{Commutative associative locus}

Assume $X=X_{\Comm}\subseteq A_W\cong W\otimes\Sym^2W^\vee$. Then
\[
  \mathrm{wt}_T(A_W)=\{\chi_i-\chi_j-\chi_k\mid 1\le i\le m, 1\le j\le k\le m\}.
\]
Moreover, $(Q_{\Comm})^\vee$ is the $G$--subrepresentation of $W\otimes(W^\vee)^{\otimes 3}$
spanned by polarized commutative associators, so its $T$--weights occur among
$\chi_i-\chi_j-\chi_k-\chi_\ell$ (cf. \cite{Quillen1968,Andre1974}).  Thus \eqref{eq:ch-identity}
again reduces to describing $\mathfrak h$, $\mathcal N(\mu)$, and $H^2_{\inc}(\mu)$ as
$H$--modules.

The Gram form is the regular trace form $\gamma_\mu(v,w)=\Tr(L_vL_w)=\Tr(L_{vw})$, and in
characteristic $0$ its nondegeneracy coincides with separability (finite \'etaleness)
\cite{DeMeyerIngraham1971}. Over an algebraically closed field, the finite \'etale locus consists
of the split algebra $\Bbbk^{\dim W}$ (cf. \cite[\S II.2]{BourbakiCA}), hence any component whose
open orbit meets it has $\rho_0(Z)=\dim W$. Away from separability, the nilradical $\mathfrak n$
satisfies $\mathfrak n\subseteq\rad(\gamma_\mu)$ because $L_x$ is nilpotent for $x\in\mathfrak n$,
and therefore $\rho_0(Z)\le \dim W-\dim\mathfrak n$.

\subsubsection*{Associative locus}

Assume $X=X_{\Ass}\subseteq A_W\cong W\otimes (W^\vee)^{\otimes 2}$. Then
\[
  \mathrm{wt}_T(A_W)=\{\chi_i-\chi_j-\chi_k\mid 1\le i,j,k\le m\},
\]
and $(Q_{\Ass})^\vee$ is spanned by polarized associators inside $W\otimes (W^\vee)^{\otimes 3}$,
so its weights occur among $\chi_i-\chi_j-\chi_k-\chi_\ell$. As before, the new input needed to
make \eqref{eq:ch-identity} explicit is the stabilizer data for $\mathfrak h$, $\mathcal N(\mu)$,
and the deformation module $H^2_{\inc}(\mu)$.

The Gram form is again a trace form, now built from right multiplications. On the separable locus
(in characteristic $0$, equivalent to semisimplicity), one has $A\cong\prod_i M_{n_i}(\Bbbk)$
\cite[\S 8]{Pierce1982}, and the trace form is nondegenerate, hence $\rho_0(Z)=\dim W$ on any
component whose open orbit meets the separable locus. If $J$ is the Jacobson radical, then $J$ is
nilpotent; for $x\in J$ one has $R_xR_y=R_{yx}$ nilpotent for all $y$, hence $\Tr(R_xR_y)=0$ and
$J\subseteq\rad(\gamma_\mu)$, giving $\rho_0(Z)\le \dim W-\dim J$.

\subsubsection*{Leibniz locus}

Assume $X=X_{\Leib}$ parametrizes right Leibniz laws. The Gram form is defined using right
multiplications $R_v(x)=\mu(x,v)$. Let
\[
  \Leib(L):=\Span\{\mu(v,v)\mid v\in W\}
\]
be the Leibniz kernel. The right Leibniz identity implies $\mu(x,\mu(v,v))=0$ for all $x,v$, i.e.
$R_z=0$ for $z\in\Leib(L)$, hence $\Leib(L)\subseteq\rad(\gamma_\mu)$ and therefore
$\rho_0(Z)\le \dim W-\dim\Leib(L)$ \cite[\S1]{Loday1993}. In particular, $\rho_0(Z)=\dim W$ forces
$\Leib(L)=0$, and over $\mathrm{char}(\Bbbk)=0$ this implies skew-symmetry, so the law is Lie; on
that locus $R_v=-\ad_v$ and $\gamma_\mu$ is the Killing form.

\begin{remark}\label{rem:weights-gram-simultaneous}
  Since $\gamma_\mu$ is $H$--fixed, it is $T$--equivariant for any maximal torus $T\subseteq H$. If
  $W=\bigoplus_\chi W_\chi$ is the $T$--weight decomposition, then $\gamma_\mu(W_\chi,W_{\chi'})=0$
  unless $\chi'=-\chi$. Thus the same weight bookkeeping that enters \eqref{eq:ch-weights} also
  constrains the $T$--character of $\rad(\gamma_\mu)$, and hence the possible values of $\rho_0(Z)$
  via \eqref{eq:gram-generic-open-orbit}.  This is often a convenient first consistency check
  before undertaking a full evaluation of \eqref{eq:ch-identity}.
\end{remark}

\subsection*{Acknowledgements}

The author acknowledges the use of large language models for the following tasks: writing the
\texttt{sagemath} code for the numerical calculations in
Section~\ref{subsec:anisotropic-example-sl2-sym14}, and copy-editing the manuscript.


\begin{thebibliography}{AOR06}

\bibitem[And74]{Andre1974}
Michel Andr{\'e}, \emph{Homologie des alg\`ebres commutatives},
  Springer-Verlag, Berlin--Heidelberg--New York, 1974.

\bibitem[AOR06]{AlbeverioOmirovRakhimov2005}
S.~Albeverio, B.~A. Omirov, and I.~S. Rakhimov, \emph{Classification of
  4-dimensional nilpotent complex {L}eibniz algebras}, Extracta Math.
  \textbf{21} (2006), no.~3, 197--210. 

\bibitem[Arz08]{Arzhantsev2008}
Ivan~V. Arzhantsev, \emph{Invariant ideals and {Matsushima}'s criterion},
  Communications in Algebra \textbf{36} (2008), no.~12, 4368--4374.

\bibitem[Aut]{stacks-project}
The Stacks~Project Authors, \emph{The stacks project},
  \url{https://stacks.math.columbia.edu}.

\bibitem[BdG13]{BurdeDeGraaf2013}
Dietrich Burde and Willem de~Graaf, \emph{Classification of {Novikov}
  algebras}, Appl. Algebra Eng. Commun. Comput. \textbf{24} (2013), no.~1,
  1--15 (English).

\bibitem[Bou89]{BourbakiCA}
Nicolas Bourbaki, \emph{Elements of mathematics. {Commutative} algebra.
  {Chapters} 1--7. {Transl}. from the {French}.}, 2nd printing ed., Berlin
  etc.: Springer-Verlag, 1989 (English).

\bibitem[Chi06]{Chipalkatti2006}
Jaydeep~V. Chipalkatti, \emph{On the invariant theory of the {B{\'e}zoutiant}},
  Beitr. Algebra Geom. \textbf{47} (2006), no.~2, 397--418 (English).

\bibitem[DI71]{DeMeyerIngraham1971}
Frank~R. DeMeyer and Edward Ingraham, \emph{Separable algebras over commutative
  rings}, Lecture Notes in Mathematics, vol. 181, Springer, Berlin, 1971.

\bibitem[Dro80]{Drozd1980}
Ju.~A. Drozd, \emph{Tame and wild matrix problems}, Representation theory {II},
  {Proc}. 2nd int. {Conf}., {Ottawa} 1979, {Lect}. {Notes} {Math}., no. 832,
  1980, pp.~242--258 (English).

\bibitem[EG98]{EdidinGraham1998}
Dan Edidin and William Graham, \emph{Equivariant intersection theory}, Invent.
  Math. \textbf{131} (1998), no.~3, 595--644.

\bibitem[FH91]{FultonHarris1991}
William Fulton and Joe Harris, \emph{Representation theory. {A} first course},
  Grad. Texts Math., vol. 129, New York etc.: Springer-Verlag, 1991 (English).

\bibitem[FP08]{FialowskiPenkava2008}
Alice Fialowski and Michael Penkava, \emph{Formal deformations, contractions
  and moduli spaces of {Lie} algebras}, International Journal of Theoretical
  Physics \textbf{47} (2008), no.~2, 561--582.

\bibitem[FP15]{FialowskiPenkava2015}
\bysame, \emph{The moduli space of 4-dimensional non-nilpotent complex
  associative algebras}, Forum Math. \textbf{27} (2015), no.~3, 1401--1434
  (English).

\bibitem[FPP11]{FialowskiPenkavaPhillipson2011}
Alice Fialowski, Michael Penkava, and Mitch Phillipson, \emph{Deformations of
  complex 3-dimensional associative algebras}, J. Gen. Lie Theory Appl.
  \textbf{5} (2011), 22, ID G110102.

\bibitem[Fuk86]{Fuks1986}
D.~B. Fuks, \emph{Cohomology of infinite-dimensional {Lie} algebras. {Transl}.
  from the {Russian} by {A}. {B}. {Sosinski{\u{\i}}}}, Contemporary {Soviet}
  {Mathematics}. {New} {York}: {Consultants} {Bureau}. xii, 339 p.; (1986).,
  1986.

\bibitem[Ful98]{Fulton1998}
William Fulton, \emph{Intersection theory}, 2 ed., Ergebnisse der Mathematik
  und ihrer Grenzgebiete. 3. Folge, vol.~2, Springer-Verlag, Berlin, 1998.

\bibitem[Gab72]{Gabriel1972}
Peter Gabriel, \emph{Unzerlegbare darstellungen {I}}, Manuscripta Mathematica
  \textbf{6} (1972), 71--103.

\bibitem[Ger64]{Gerstenhaber1964}
Murray Gerstenhaber, \emph{On the deformation of rings and algebras}, Annals of
  Mathematics \textbf{79} (1964), no.~1, 59--103.

\bibitem[GK96]{GozeKhakimdjanov1996}
Michel Goze and Yury Khakimdjanov, \emph{Nilpotent {Lie} algebras}, Mathematics
  and its Applications, vol. 361, Kluwer Academic Publishers, Dordrecht, 1996.

\bibitem[GM88]{GoldmanMillson1988}
William~M. Goldman and John~J. Millson, \emph{The deformation theory of
  representations of fundamental groups of compact {K}\"ahler manifolds}, Inst.
  Hautes \'Etudes Sci. Publ. Math. \textbf{67} (1988), 43--96.

\bibitem[GO88]{GrunewaldOHalloran1988}
Fritz~J. Grunewald and John O'Halloran, \emph{Varieties of nilpotent {Lie}
  algebras of dimension less than six}, Journal of Algebra \textbf{112} (1988),
  no.~2, 315--325.

\bibitem[Gon98]{Gong1998}
M.~P. Gong, \emph{Classification of nilpotent {Lie} algebras of dimension 7
  (and 6)}, Ph.D. thesis, North Carolina State University, 1998.

\bibitem[Har62]{Harrison1962}
David~K. Harrison, \emph{Commutative algebras and cohomology}, Transactions of
  the American Mathematical Society \textbf{104} (1962), 191--204.

\bibitem[HNP19]{HarpazNuitenPrasma2019}
Yonatan Harpaz, Joost Nuiten, and Matan Prasma, \emph{Tangent categories of
  algebras over operads}, Israel Journal of Mathematics \textbf{234} (2019),
  691--742, arXiv:1612.02607.

\bibitem[Hum72]{Humphreys1972}
James~E. Humphreys, \emph{Introduction to lie algebras and representation
  theory}, Graduate Texts in Mathematics, vol.~9, Springer, 1972.

\bibitem[Kay25]{Kaygun2025}
Atabey Kaygun, \emph{Geometry of deformations via incidence varieties}, {\tt
  arXiv:2511.17169}, 2025.

\bibitem[KV19]{KaygorodovVolkov2019}
Ivan Kaygorodov and Yury Volkov, \emph{The variety of two-dimensional algebras
  over an algebraically closed field}, Canadian Journal of Mathematics
  \textbf{71} (2019), no.~4, 819--842.

\bibitem[Lod93]{Loday1993}
Jean-Louis Loday, \emph{Une version non commutative des alg\`ebres de {Lie}:
  les alg\`ebres de {Leibniz}}, L'Enseignement Math\'ematique \textbf{39}
  (1993), no.~3--4, 269--293.

\bibitem[LP93]{LodayPirashvili1993}
Jean-Louis Loday and Teimuraz Pirashvili, \emph{Universal enveloping algebras
  of leibniz algebras and (co)homology}, Mathematische Annalen \textbf{296}
  (1993), 139--158.

\bibitem[Mar13]{Martin2012}
Mar{\'{\i}}a~Eugenia Martin, \emph{Four dimensional {Jordan} algebras}, Int. J.
  Math. Game Theory Algebra \textbf{20} (2013), no.~4, 303--321 (English).

\bibitem[Mat60]{Matsushima1960}
Yoz\^o Matsushima, \emph{Espaces homog\`enes de {S}tein des groupes de {Lie}
  complexes}, Nagoya Mathematical Journal \textbf{16} (1960), 205--218.

\bibitem[Maz79]{Mazzola1979}
Guerino Mazzola, \emph{The algebraic and geometric classification of
  associative algebras of dimension five}, Manuscr. Math. \textbf{27} (1979),
  81--101.

\bibitem[NR67]{NijenhuisRichardson1967}
Albert Nijenhuis and Roger~W. Richardson, \emph{Deformations of {Lie} algebra
  structures}, Journal of Mathematics and Mechanics \textbf{17} (1967),
  89--105.

\bibitem[Pie82]{Pierce1982}
Richard~S. Pierce, \emph{Associative algebras}, Graduate Texts in Mathematics,
  vol.~88, Springer, 1982.

\bibitem[Qui70]{Quillen1968}
Daniel~G. Quillen, \emph{On the (co)-homology of commutative rings},
  Applications of Categorical Algebra (Proc.\ Sympos.\ Pure Math., Vol.\ XVII,
  New York, 1968), American Mathematical Society, Providence, RI, 1970,
  pp.~65--87.

\bibitem[Ric67]{Richardson1967}
R.~W.~jun. Richardson, \emph{On the rigidity of semi-direct products of {Lie}
  algebras}, Pac. J. Math. \textbf{22} (1967), 339--344 (English).

\bibitem[Sch68]{Schlessinger1968}
Michael Schlessinger, \emph{Functors of {A}rtin rings}, Transactions of the
  American Mathematical Society \textbf{130} (1968), 208--222.

\bibitem[SGA6]{SGA6}
Alexander Grothendieck et~al., \emph{Th\'eorie des intersections et
  th\'eor\`eme de {R}iemann--{R}och}, S\'eminaire de G\'eom\'etrie Alg\'ebrique
  du {B}ois {M}arie 1966--67 ({SGA} 6), Lecture Notes in Mathematics, vol. 225,
  Springer, 1971.
  
\bibitem[Tho87]{Thomason1987}
R.~W. Thomason, \emph{Algebraic {$K$}-theory of group scheme actions},
  Algebraic topology and algebraic {$K$}-theory ({P}rinceton, {N}.{J}., 1983),
  Ann. of Math. Stud., vol. 113, Princeton Univ. Press, Princeton, NJ, 1987,
  pp.~539--563. 

\bibitem[Tho92]{Thomason1992}
Robert~W. Thomason, \emph{Une formule de {L}efschetz en {$K$}-th\'eorie
  \'equivariante alg\'ebrique}, Duke Mathematical Journal \textbf{68} (1992),
  no.~3, 447--462.

\bibitem[Wei13]{WeibelKBook}
Charles~A. Weibel, \emph{The {$K$}-book: An introduction to algebraic
  {$K$}-theory}, Graduate Studies in Mathematics, vol. 145, American
  Mathematical Society, 2013.

\end{thebibliography}

\newcommand{\etalchar}[1]{$^{#1}$}

\end{document}